\documentclass[12pt]{article}
\usepackage{amsfonts}
\usepackage{amssymb}
\usepackage{amsbsy}
\usepackage{amsthm}
\usepackage{amsmath}
\usepackage{amscd}
\usepackage[usenames,dvipsnames]{xcolor}
\usepackage{mathrsfs}
\usepackage{holtpolt}
\usepackage{mathabx}
\usepackage{ifsym}
\usepackage{ulem}
\usepackage{graphicx}
\usepackage{mathtools}
\usepackage[all]{xy}
\newtheorem{thm}{Theorem}[section]
\newtheorem{cor}[thm]{Corollary}
\newtheorem{lem}[thm]{Lemma}
\newtheorem{prop}[thm]{Proposition}
\theoremstyle{definition}
\newtheorem{defn}[thm]{Definition}
\theoremstyle{remark}
\newtheorem{rem}[thm]{Remark}
\newtheorem{ex}[thm]{\bf Example}
\newcommand{\bt}{\begin{thm}}
\newcommand{\et}{\end{thm}}
\newcommand{\bc}{\begin{cor}}
\newcommand{\ec}{\end{cor}}
\newcommand{\bl}{\begin{lem}}
\newcommand{\el}{\end{lem}}
\newcommand{\bp}{\begin{prop}}
\newcommand{\ep}{\end{prop}}
\newcommand{\bd}{\begin{defn}}
\newcommand{\ed}{\end{defn}}
\newcommand{\br}{\begin{rem}}
\newcommand{\er}{\end{rem}}
\newcommand{\bpr}{\begin{proof}}
\newcommand{\epr}{\end{proof}}
\newcommand{\bex}{\begin{ex}}
\newcommand{\eex}{\end{ex}}
\newcommand{\bcd}{\begin{CD}}
\newcommand{\ecd}{\end{CD}}

\newcommand{\bi}{\begin{itemize}}
\newcommand{\ei}{\end{itemize}}
\newcommand{\be}{\begin{enumerate}}
\newcommand{\ee}{\end{enumerate}}
\newcommand{\ba}{\begin{array}}
\newcommand{\ea}{\end{array}}
\newcommand{\beq}{\begin{equation}}
\newcommand{\eeq}{\end{equation}}
\newcommand{\beqa}{\begin{eqnarray}}
\newcommand{\eeqa}{\end{eqnarray}}
\newcommand{\bca}{\begin{cases}}
\newcommand{\eca}{\end{cases}}
\newcommand{\bal}{\begin{aligned}}
\newcommand{\eal}{\end{aligned}}

\newcommand{\Z}{{\mathbb Z}}
\newcommand{\R}{{\mathbb R}}
\newcommand{\C}{{\mathbb C}}
\newcommand{\T}{{\mathbb T}}

\newcommand{\DVZ}{{\operatorname{DVZ}}}
\newcommand{\DG}{{\operatorname{DG}}}
\newcommand{\Sz}{{\operatorname{Sz}}}
\newcommand{\spn}{{\operatorname{span}}}

\newcommand{\re}{{\operatorname{Re}}}
\newcommand{\im}{{\operatorname{Im}}}

\newcommand{\sgn}{{\operatorname{sgn}}}

\newcommand{\<}{\langle}
\renewcommand{\>}{\rangle}
\newcommand{\uk}{|\kern-3.3pt\uparrow\>}
\newcommand{\dk}{|\kern-3.3pt\downarrow\>}
\newcommand{\ub}{\<\uparrow\kern-3.3pt|}
\newcommand{\db}{\<\downarrow\kern-3.3pt|}

\begin{document}

%\layout 

\title{\bf A CMV connection between \break orthogonal polynomials on the unit circle and the real line}
\author{M. J. Cantero$^1$, F. Marcell\'an$^2$, L. Moral$^1$, L. Vel\'azquez$^1$
\footnote{The work of the first, third and fourth authors has been supported in part by the research project MTM2017-89941-P from Ministerio de Econom\'{\i}a, Industria y Competitividad of Spain and the European Regional Development Fund (ERDF), by project UAL18-FQM-B025-A (UAL/CECEU/FEDER) and by project E26\_17R of Diputaci\'on General de Arag\'on (Spain) and the ERDF 2014-2020 ``Construyendo Europa desde Arag\'on". The work of the second author has been partially supported by the research project PGC2018--096504-B-C33 from Agencia Estatal de investigaci\'on  of Spain.}}
\date{\small $^1$ Departamento de Matem\'atica Aplicada and IUMA, Universidad de Zaragoza, Spain
\break
$^2$ Departamento de Matem\'aticas, Universidad Carlos III de Madrid, Spain.}

\maketitle

\begin{abstract}

M.~Derevyagin, L.~Vinet and A.~Zhedanov introduced in \cite{DVZ} a new connection between orthogonal polynomials on the unit circle and the real line. It maps any real CMV matrix into a Jacobi one depending on a real parameter $\lambda$. In \cite{DVZ} the authors prove that this map yields a natural link between the Jacobi polynomials on the unit circle and the little and big $-1$ Jacobi polynomials on the real line. They also provide explicit expressions for the measure and orthogonal polynomials associated with the Jacobi matrix in terms of those related to the CMV matrix, but only for the value $\lambda=1$ which simplifies the connection --{\it basic DVZ connection}--. However, similar explicit expressions for an arbitrary value of $\lambda$ --{\it (general) DVZ connection}-- are missing in \cite{DVZ}. This is the main problem overcome in this paper.

This work introduces a new approach to the DVZ connection which formulates it as a two-dimensional eigenproblem by using known properties of CMV matrices. This allows us to go further than \cite{DVZ}, providing explicit relations between the measures and orthogonal polynomials for the general DVZ connection. It turns out that this connection maps a measure on the unit circle into a rational perturbation of an even measure supported on two symmetric intervals of the real line, which reduce to a single interval for the basic DVZ connection, while the perturbation becomes a degree one polynomial. Some instances of the DVZ connection are shown to give new one-parameter families of orthogonal polynomials on the real line.

\end{abstract}


\medskip

\noindent{\it (2010) AMS Mathematics Subject Classification}: 42C05.

%%%%%%%%%%%%%%%%%%%%%%%%%%%%%%%%%%%%%%%%%%%%%%%%%%%%%%%%%%%%%%%%%%%%%%%%%%
\section {Introduction}
\label{sec:Int}
%%%%%%%%%%%%%%%%%%%%%%%%%%%%%%%%%%%%%%%%%%%%%%%%%%%%%%%%%%%%%%%%%%%%%%%%%%

Any sequence $p_n$ of orthonormal polynomials with respect to a measure on the real line (OPRL) is characterized by a Jacobi matrix,
\begin{equation} \label{eq:Jac}
 {\cal J} =
 \begin{pmatrix}
	b_0 & a_0
	\\
	a_0 & b_1 & \kern-2pt a_1
	\\
    & a_1 & \kern-2pt b_2 & \kern-3pt a_2
    \\
    & & \kern-2pt \ddots & \kern-3pt \ddots & \kern-3pt \ddots
 \end{pmatrix},
 \qquad b_n\in\R, \qquad a_n>0,
\end{equation}
encoding the corresponding three term recurrence relation,
$$
 {\cal J}p(x) = xp(x), \qquad\quad p=(p_0,p_1,p_2,\dots)^t.
$$
Without loss of generality, we can suppose the measure normalized so that $p_0=1$.

When the measure is supported on the unit circle $\T=\{z\in\C : |z|=1\}$, another kind of recurrence relation characterizes the corresponding sequences of orthogonal polynomials (OPUC). In the case of monic OPUC $\phi_n$, it has the form \cite{Ger,GrSz,SiOPUC,Sz}
\begin{equation} \label{eq:mRR}
 \phi_{n+1}(z) = z\phi_n(z) - \overline{\alpha_n}\phi_n^*(z),
 \qquad \alpha_n \in \C, \qquad |\alpha_n|<1,
\end{equation}
where $\phi_n^*(z)=z^n\overline{\phi_n(1/\overline{z})}$ is known as the reversed polynomial of $\phi_n$ and the parameters $\alpha_n$ are called Verblunsky coefficients. Introducing the complementary parameters $\rho_n = \sqrt{1-|\alpha_n|^2}$, the orthonormal OPUC $\varphi_n$ with positive leading coefficients become
$$
 \varphi_n(z) = \kappa_n\phi_n(z) =
 \kappa_n(-\overline{\alpha_{n-1}}+\cdots+z^n),
 \qquad
 \kappa_n=(\rho_0\rho_1\cdots\rho_{n-1})^{-1}.
$$
A more natural set of funtions is constituted by the orthonormal Laurent polynomials on the unit circle (OLPUC), given by
\begin{equation} \label{eq:OLPUC-OPUC}
 \begin{aligned}
 	& \chi_{2k}(z) = z^{-k} \varphi_{2k}^*(z)
 	= z^k \overline{\varphi_{2k}(1/\overline{z})}
 	= \kappa_{2k} (-\alpha_{2k-1}z^k+\cdots+z^{-k}),
 	\\
 	& \chi_{2k+1}(z) = z^{-k} \varphi_{2k+1}(z)
 	= \kappa_{2k+1} (z^{k+1}+\cdots+(-\overline{\alpha_{2k}})z^{-k}).
 \end{aligned}
\end{equation}
They provide a simple matrix representation of the recurrence relation (see \cite{CMV03,SiOPUC,Watkins}), which reads as
$$
 {\cal C}\chi(z) = z\chi(z), \qquad\quad \chi=(\chi_0,\chi_1,\chi_2,\dots)^t,
$$
in terms of the so called CMV matrix ${\cal C}$ --the unitary analogue of a Jacobi matrix \cite{KN07,SiOPUC}--, a five-diagonal unitary matrix which factorizes as ${\cal C} = {\cal M}{\cal L}$, with
\begin{equation} \label{eq:LMfac}
 {\cal L} =
 \begin{pmatrix}
 	\Theta_0
 	\\[-2pt]
 	& \kern-3pt \Theta_2
 	\\[-2pt] 	
	& & \kern-3pt \Theta_4
 	\\[-2pt]
 	& & & \kern-3pt \ddots
 \end{pmatrix},
 \kern15pt
 {\cal M} =
 \begin{pmatrix}
 	1
 	\\[-2pt]
 	& \kern-3pt \Theta_1
 	\\[-2pt]
 	& & \kern-3pt \Theta_3
 	\\[-2pt]
 	& & & \kern-3pt \ddots
 \end{pmatrix},
 \kern15pt
 \Theta_n =
 \begin{pmatrix}
 	\overline{\alpha_n} & \rho_n
 	\\
 	\rho_n & -\alpha_n
 \end{pmatrix}.
\end{equation}
We refer to this as the $\Theta$-factorization of a CMV matrix. Analogously to the case of the real line, we will assume that $\chi_0=\varphi_0=1$ by normalizing the measure.

Since this paper revolves around a new connection between OPUC and OPRL, it is worth commenting on known ones. The paradigm of such connections, due to Szeg\H o \cite{Ger,SiOPUC,Sz}, starts with a measure $\mu$ on $\T$ which is symmetric under conjugation, which means that the Verblunsky coefficients are real or, in other words, the related OPUC $\varphi_n$ have real coefficients. The measure $\sigma$ induced on $[-2,2]$ by the mapping $z \mapsto x=z+z^{-1}$ is called its Szeg\H o projection, which we denote by $\sigma=\Sz(\mu)$. The related OPRL are given by \cite{SiOPUC,Sz}
\begin{equation} \label{eq:SzOPRL}
 p_n(x) = z^{-n}
 \frac{\varphi_{2n}(z)+\varphi_{2n}^*(z)}
 {\sqrt{2(1-\alpha_{2n-1})}},
 \qquad
 x=z+z^{-1}.
\end{equation}
Here and in what follows we use the convention $\alpha_{-1}=-1$.

A more recent OPUC-OPRL connection goes back to works of Delsarte and Genin in the framework of the split Levinson algorithm \cite{DG,DG1,DG2}. They realized that
\begin{equation} \label{eq:DGOPRL}
 \hat{p}_n(x) = (z^{1/2})^{-n}
 \frac{\varphi_n(z)+\varphi_n^*(z)}
 {\sqrt{2(1-\alpha_{n-1})}},
 \qquad
 x=z^{1/2}+z^{-1/2},
\end{equation}
constitute an OPRL sequence with respect to an even measure on $[-2,2]$. We will refer to this as the DG connection between OPUC and OPRL.

Although the relations between the measures and orthogonal polynomials are simple for the above two connections, the situation is somewhat different regarding the CMV and Jacobi matrices: for both connections, the relations expressing the Jacobi parameters in terms of the Verblunsky coefficients are non-trivial to invert \cite{DG,SiOPUC,Sz}. This changes completely with the new connection (see Proposition~\ref{prop:J-DVZ}), directly defined by a very simple relation between CMV and Jacobi matrices (see \cite{DVZ} and Section~\ref{sec:DVZ}). We refer to this by the surname initials of its discoverers, Derevyagin, Vinet and Zhedanov: the DVZ connection.

DVZ has other properties which distinguish it from Szeg\H o and DG. First, for each OPUC instance with real coefficients, instead of a single OPRL sequence, it gives a family of OPRL depending on an arbitrary real parameter $\lambda$. Besides, up to $\lambda=\pm1$, the DVZ measure on the real line is supported on two disjoint symmetric intervals, although it is not even, in contrast to the DG measure on $[-2,2]$. The price to pay for these differences is a slightly more involved relation between the OPUC and OPRL linked by DVZ, which includes a less trivial relation between the unit circle and real line variables than for Szeg\H o and DG.

The rest of the paper is devoted to the DVZ connection, which is detailed in the following way: Section~\ref{sec:DVZ} describes two especially simple cases of the DVZ connection, corresponding $\lambda=\pm1$. The aim is to introduce a new eigenproblem translation of DVZ in the easiest possible setting. Section~\ref{sec:SDs} reconciles this new approach with the original one in \cite{DVZ}, in which the basic DVZ connection for $\lambda=1$ arises as a composition of DG and a Christoffel transformation on the real line \cite{BuMa,Zhe}, i.e. the multiplication of a measure on the real line by a real polynomial. The main results of the paper are in Section~\ref{sec:gDVZ} where, extending the previous eigenmachinery to the general DVZ connection given by an arbitrary $\lambda\in\R$, we obtain explicit expressions for the corresponding measures and orthogonal polynomials on the real line in terms of the unit circle counterparts. This is shown to provide new examples of one-parameter OPRL families in Section~\ref{sec:Ex}.

%%%%%%%%%%%%%%%%%%%%%%%%%%%%%%%%%%%%%%%%%%%%%%%%%%%%%%%%%%%%%%%%%%%%%%%%%%
\section {A new approach to the basic DVZ connection}
\label{sec:DVZ}
%%%%%%%%%%%%%%%%%%%%%%%%%%%%%%%%%%%%%%%%%%%%%%%%%%%%%%%%%%%%%%%%%%%%%%%%%%

Derevyagin, Vinet and Zhedanov established in \cite{DVZ} a new connection between OPUC and OPRL, closely related to the DG connection. The starting point for this was the $\Theta$-factorization ${\cal C}={\cal M}{\cal L}$ of a CMV matrix, given by the tridiagonal factors \eqref{eq:LMfac}. If the Verblunsky coefficients are real, a Jacobi matrix ${\cal K}$ can be built up out of the symmetric unitary factors ${\cal L}$ and ${\cal M}$,
\begin{equation} \label{eq:newL+M}
 {\cal K} = {\cal L}+{\cal M} =
 \begin{pmatrix}
 	\alpha_0+1 & \rho_0
 	\\[2pt]
 	\rho_0 & \alpha_1-\alpha_0 & \rho_1
 	\\[2pt]
 	& \rho_1 & \alpha_2-\alpha_1 & \rho_2
 	\\[2pt]
	& & \rho_2 & \alpha_3-\alpha_2 & \kern15pt \rho_3
 	\\
 	& & & \kern-35pt \ddots & \kern-30pt \ddots & \ddots
 \end{pmatrix}.
\end{equation}
In \cite{DVZ}, Derevyagin, Vinet and Zhedanov identified the measure and OPRL related to ${\cal K}$ in terms of the measure and OPUC associated with ${\cal C}$. Such an identification appeared surprisingly when combining the DG connection with a Christoffel transformation on the real line. This defines what we will call the basic DVZ connection (a more general one will come later on), which maps OPUC with real coefficients into OPRL with respect to a measure supported on $[-2,2]$. We will refer to $\mathcal{K}$ as the basic DVZ transform of $\mathcal{C}$.

Since ${\cal L}^2={\cal M}^2=I$ for $\alpha_n\in\R$ and $|\alpha_n|<1$, the Jacobi matrix \eqref{eq:newL+M} satisfies
$$
 {\cal K}^2-2I={\cal C}+{\cal C}^\dag.
$$
where $\dag$ denotes the adjoint of a matrix. This identity encodes the relation between the basic DVZ connection and the Szeg\H o projection --whose Jacobi matrix is hidden in ${\cal C}+{\cal C}^\dag$ \cite{KN04,SiOPUC}--, and suggests also a link with symmetrization processes on $[-2,2]$ --adressed by the mapping $x \mapsto x^2-2$ (see Section~\ref{sec:SDs}), which is represented by the operation ${\cal K}^2-2I$ on the Jacobi matrix ${\cal K}$--. Actually, we will show that the basic DVZ connection follows by a concatenation of a symmetrization process \cite{MaSa} and a Christoffel transformation \cite{CMMV,DaHeMa,GaMa} on the unit circle, followed by the Szeg\H o projection. Equivalently, we can implement first the Szeg\H o mapping, and then perform the real line version of the symmetrization (see Section~\ref{sec:SDs}) and Christoffel \cite{BuMa,Zhe}.

In this section we will present the basic DVZ connection in a more natural and concise way, directly looking for the relation between the OLPUC related to ${\cal C}$ and the OPRL associated with the basic DVZ transform ${\cal K}$, and then using this to uncover the relation between the corresponding orthogonal polynomials and orthogonality measures. The interpretation in terms of symmetrizations, Szeg\H o  and Christoffel will follow as a byproduct of these results.

The OPRL $q_n$ corresponding to the Jacobi matrix ${\cal K}$ are the solutions of the formal eigenvalue equation
$$
 {\cal K} q(x) = x q(x),
 \qquad\quad
 q=(q_0,q_1,q_2,\dots)^t, \qquad\quad q_0=1.
$$
We will identify these OPRL by using some relations, often not fully exploited, between the OLPUC $\chi_n$ related to $\cal C$ and the factors $\cal L$, $\cal M$ \cite{CMV03},
$$
 {\cal L}\chi(z) = z\chi_*(z),
 \qquad\quad
 {\cal M}\chi_*(z)=\chi(z),
 \qquad\quad
 \begin{cases}
 	\chi=(\chi_0,\chi_1,\chi_2,\dots)^t,
 	\\
	\chi_*=(\chi_{0*},\chi_{1*},\chi_{2*},\dots)^t.
 \end{cases}
$$
Bearing in mind that ${\cal L}$ and ${\cal M}$ are involutions for $\alpha_n\in(-1,1)$, the above relations can be equivalently written as
$$
 {\cal L}\chi_*(z) = z^{-1}\chi(z),
 \qquad\quad
 {\cal M}\chi(z)=\chi_*(z).
$$
Therefore,
$$
 {\cal K}\chi(z)=(z+1)\chi_*(z),
 \qquad\quad
 {\cal K}\chi_*(z)=(z^{-1}+1)\chi(z),
$$
which show that, for each fixed $z\in\C^*=\C\setminus\{0\}$, the linear subspace $\spn\{\chi(z),\chi_*(z)\}$ is invariant for ${\cal K}$. This implies that ${\cal K}$ has formal eigenvectors $u(z)\chi(z)+v(z)\chi_*(z)$ given by the eigenvalue equation
$$
 \begin{pmatrix}
 0 & z^{-1}+1
 \\
 z+1 & 0
 \end{pmatrix}
 \begin{pmatrix} u(z) \\ v(z) \end{pmatrix} = x
 \begin{pmatrix} u(z) \\ v(z) \end{pmatrix}.
$$
The eigenvalues $x$ are the solutions of
$$
 x^2 = (z+1)(z^{-1}+1),
$$
thus we can write
$$
 x = z^{1/2}+z^{-1/2}
$$
for a choice of the square root $z^{1/2}$ (here and in what follows we understand that $z^{n/2}=(z^{1/2})^n$, $n\in\Z$, for that choice of the square root). The corresponding eigenvectors of ${\cal K}$ are spanned by
\begin{equation} \label{eq:X}
 X(z) = \chi(z) + \frac{z+1}{x} \chi_*(z) =
 \chi(z) + z^{1/2} \chi_*(z).
\end{equation}
We conclude that the OPRL $q_n$ are given by
\begin{equation} \label{eq:q}
\begin{gathered}
 q_n(x) = \frac{X_n(z)}{X_0(z)}
 = \frac{\chi_n(z)+z^{1/2}\chi_{n*}(z)}{1+z^{1/2}}
 = z^{-n/2} \, \frac{\varphi_n^*(z)+z^{1/2}\varphi_n(z)}{1+z^{1/2}},
 \\[2pt]
 x=z^{1/2}+z^{-1/2},
\end{gathered}
\end{equation}
where we have used \eqref{eq:OLPUC-OPUC}. This coincides with the expression given in \cite{DVZ}.

The orthogonality measure $\nu$ on the real line for $q_n$ follows from that one $\mu$ on the unit circle for $\chi_n$. To obtain this relation, let us parametrize the unit circle as $\T=\{e^{i\theta}:\theta\in[0,2\pi)\}$ and consider $\mu$ as a measure on $[0,2\pi]$, where the mass at $1\in\T$, if any, is distributed equally between $\theta=0$ and $\theta=2\pi$ so that $\mu$ remains symmetric under the transformation $\theta \mapsto 2\pi-\theta$, which represents the conjugation of $e^{i\theta}$. Then, the restriction of $x=z^{1/2}+z^{-1/2}$ to the unit circle leads to the map
\begin{equation} \label{eq:x}
 x(\theta)=2\cos\frac{\theta}{2},
 \qquad \theta\in[0,2\pi],
\end{equation}
which is a one-to-one transformation between $[0,2\pi]$ and $[-2,2]$ with inverse
\begin{equation} \label{eq:th}
 \theta(x) = 2 \arccos\frac{x}{2},
 \qquad x\in[-2,2].
\end{equation}
The symmetric extension of $\mu$ to $[0,2\pi]$ is crucial to identify the orthogonality measure of $q_n$, and guarantees that the relation $x(2\pi-\theta)=-x(\theta)$ makes sense for any angle $\theta$ in the interval $[0,2\pi]$ where $\mu$ is supported.

The orthogonality measure of $q_n(x)=X_n(z)/X_0(z)$ arises from the observation that the components $X_n$ of the formal eigenvector \eqref{eq:X} are orthogonal with respect to $\mu$. Indeed,
$$
\begin{aligned}
 \int_0^{2\pi} X(e^{i\theta}) X(e^{i\theta})^\dag \, d\mu(\theta)
 & = \int_0^{2\pi}
 (\chi(e^{i\theta}) \chi(e^{i\theta})^\dag +
 \chi_*(e^{i\theta}) \chi_*(e^{i\theta})^\dag)
 \, d\mu(\theta)
 \\
 & + \int_0^{2\pi}
 (e^{i\theta/2} \chi_*(e^{i\theta}) \chi(e^{i\theta})^\dag +
 e^{-i\theta/2} \chi(e^{i\theta}) \chi_*(e^{i\theta})^\dag)
 \, d\mu(\theta).
\end{aligned}
$$
Due to the orthonormality of $\chi_n$ and $\chi_{n*}$ with respect to $\mu$, the first term in the right hand side is 2I. On the other hand, taking into account that the OLPUC have real coefficients due to the symmetry of the measure under conjugation, the second term vanishes because its integrand simply changes sign under the transformation $\theta \mapsto 2\pi-\theta$ which leaves the measure $\mu$ invariant. Therefore,
$$
 \int_{-2}^2 q(x) q(x)^\dag \left|X_0(e^{i\theta(x)})\right|^2 d\mu(\theta(x)) =
 \int_0^{2\pi} X(e^{i\theta}) X(e^{i\theta})^\dag \, d\mu(\theta) = 2I,
$$
which proves that the orthogonality measure of $q_n$ is
\begin{equation} \label{eq:nu}
 d\nu(x) = \frac{1}{2} \left|1+e^{i\theta(x)/2}\right|^2 d\mu(\theta(x)) =
 \frac{1}{2} (2+x) \, d\mu(\theta(x)),
 \qquad x\in[-2,2],
\end{equation}
a result also present in \cite{DVZ}.

Since $\theta(-x)=2\pi-\theta(x)$, the symmetry of $\mu$ under $\theta \mapsto 2\pi-\theta$ makes $d\mu(\theta(x))$ symmetric under $x \mapsto -x$. However, the orthogonality measure $\nu$ does not preserve finally this symmetry due to the additional factor $\left|X_0(e^{i\theta(x)/2})\right|^2=2+x$ that comes from the relation $d\nu(x)=\frac{1}{2}\left|X_0(e^{i\theta(x)/2})\right|^2\,d\mu(\theta(x))$.

\medskip

The above discussion may be extended to the tridiagonal matrix
\begin{equation} \label{eq:L-M}
 {\cal K}_{-} = {\cal L}-{\cal M} =
 \begin{pmatrix}
 	\alpha_0-1 & \rho_0
 	\\[2pt]
 	\rho_0 & -\alpha_1-\alpha_0 & -\rho_1
 	\\[2pt]
 	& -\rho_1 & \alpha_2+\alpha_1 & \rho_2
 	\\[2pt]
	& & \rho_2 & -\alpha_2-\alpha_3 & \kern15pt -\rho_3
 	\\
 	& & & \kern-35pt \ddots & \kern-30pt \ddots & \ddots
 \end{pmatrix},
\end{equation}
which is the conjugated of a Jacobi matrix by a diagonal sign matrix,
\begin{equation} \label{eq:Pi}
 \kern-9pt
 \Pi {\cal K}_{-} \Pi \kern-1pt = \kern-2pt
 \left(
 \begin{smallmatrix}
    \\
 	\alpha_0-1 & \rho_0
 	\\[5pt]
 	\rho_0 & -\alpha_1-\alpha_0 & \kern3pt \rho_1
 	\\[5pt]
 	& \rho_1 & \kern3pt \alpha_2+\alpha_1 & \rho_2
 	\\[5pt]
 	& & \kern3pt \rho_2 & -\alpha_3-\alpha_2 & \kern12pt \rho_3
 	\\[-2pt]
 	& & & \kern-20pt \ddots & \kern-20pt \ddots & \ddots
	\\[1pt]
 \end{smallmatrix}
 \right)\kern-1pt,
 \quad
 \Pi \kern-1pt = \kern-2pt
 \left(
 \begin{smallmatrix}
 	\\[2pt]
 	1 \\[1pt]
 	& \kern4pt 1 \\[1pt]
 	& & \kern-2pt -1 \\[1pt]
 	& & & \kern-2pt -1 \\[1pt]
 	& & & & \kern4pt 1 \\[1pt]
 	& & & & & \kern4pt 1 \\[1pt]
 	& & & & & & \kern-2pt -1 \\[1pt]
 	& & & & & & & \kern-2pt -1 \\[-5pt]
 	& & & & & & & & \ddots
 \\[2pt]
 \end{smallmatrix}
 \right)\kern-1pt.
\end{equation}
Therefore, the formal eigenvalue equation
$$
 {\cal K}_- q^-(x) = x q^-(x),
 \qquad\quad
 q^-=(q^-_0,q^-_1,q^-_2,\dots)^t, \qquad\quad q^-_0=1,
$$
defines a sequence $q^-_n$ of OPRL whose leading coefficients have a sign given by the corresponding diagonal coefficient of $\Pi$.

A direct translation of the previous reasoning to this case shows that ${\cal K}_-$ has formal eigenvectors
$$
 \chi(z) + iz^{1/2} \chi_*(z)
$$
with eigenvalue
$$
 x = -i(z^{1/2}-z^{-1/2}),
$$
so that
$$
\begin{gathered}
 q^-_n(x)
 = \frac{\chi_n(z)+iz^{1/2}\chi_{n*}(z)}{1+iz^{1/2}}
 = z^{-n/2} \, \frac{\varphi_n^*(z)+iz^{1/2}\varphi_n(z)}{1+iz^{1/2}},
 \\[2pt]
 x=-i(z^{1/2}-z^{-1/2}).
\end{gathered}
$$

As for the orthogonality measure $\nu_-$ of $q^-_n$, in this case a convenient parametrization of the unit circle is $\T=\{e^{i\theta}:\theta\in(-\pi,\pi]\}$. Accordingly, we consider the measure $\mu$ of $\chi_n$ as a measure on $[-\pi,\pi]$, with any possible mass at $-1\in\T$ distributed equally between $\theta=-\pi$ and $\theta=\pi$ to preserve the symmetry of $\mu$ under the transformation $\theta \mapsto -\theta$ representing the conjugation of $e^{i\theta}$. Then, on the unit circle, $x=-i(z^{1/2}+z^{-1/2})$ becomes
\begin{equation} \label{eq:x-}
 x(\theta)=2\sin\frac{\theta}{2},
 \qquad \theta\in[-\pi,\pi],
\end{equation}
which maps $[-\pi,\pi]$ one-to-one onto $[-2,2]$, and has the inverse
\begin{equation} \label{eq:th-}
 \theta(x) = 2 \arcsin\frac{x}{2},
 \qquad x\in[-2,2].
\end{equation}
The same kind of argument as in the previous case proves that
$$
 d\nu_-(x) = \frac{1}{2} \left|1+ie^{i\theta(x)/2}\right|^2 d\mu(\theta(x))
 = \frac{1}{2} (2-x) \, d\mu(\theta(x)),
 \qquad x\in[-2,2].
$$
Bearing in mind that $\theta(-x)=-\theta(-x)$, the measure $\nu_-$ fails to be symmetric with respect to $x \mapsto -x$ only due to the factor $2-x$.

%\medskip

%In the discussion concerning both, ${\cal K}$ and ${\cal K}_-$, it would have been enough to consider, instead of the corresponding mappings $z \mapsto x$ on the complex plane, their restrictions $x(\theta)$ and $x_-(\theta)$ to the unit circle. This is obvious regarding the orthogonality measures, while concerning the OPRL it is a consequence of the fact that they are completely characterized by their values on an interval of the real line, as it is the case for instance of the orthogonality interval $[-2,2]$. We will follow this approach in Sect.~\ref{sec:gDVZ}, which generalizes the DVZ connection to tridiagonal matrices given by arbitrary real linear combinations of ${\cal L}$ and ${\cal M}$.

The next section will present a closer look at the relations between the basic DVZ connection and the well known Szeg\H o projection between the unit circle and the real line. This will make symmetrization processes and Christoffel transformations to enter into the game.

%%%%%%%%%%%%%%%%%%%%%%%%%%%%%%%%%%%%%%%%%%%%%%%%%%%%%%%%%%%%%%%%%%%%%%%%%%
\section{DVZ, Christoffel, Szeg\H o and symmetrizations}
\label{sec:SDs}
%%%%%%%%%%%%%%%%%%%%%%%%%%%%%%%%%%%%%%%%%%%%%%%%%%%%%%%%%%%%%%%%%%%%%%%%%%

Originally, in \cite{DVZ}, the basic DVZ connection arises as a combination of well known transformations which surprisingly links the OPUC related to a CMV matrix $\mathcal{C}=\mathcal{M}\mathcal{L}$ and the OPRL associated to the Jacobi matrix $\mathcal{K}=\mathcal{L}+\mathcal{M}$. We will see that such an interpretation follows directly from our more direct approach. For this purpose, let us rewrite the DVZ relation \eqref{eq:q} between OPUC and OPRL as
$$
 q_n(x) =
 z^{-n} \, \frac{\varphi_n^*(z^2)+z\varphi_n(z^2)}{1+z},
 \qquad
 x=z+z^{-1}.
$$
This relation suggests introducing the symmetrized OPUC \cite{MaSa},
\begin{equation} \label{eq:symOPUC}
 \hat\varphi_n(z) =
 \begin{cases}
 	\varphi_k(z^2), & n=2k,
	\\
	z\varphi_k(z^2), & n=2k+1,
 \end{cases}
\end{equation}
characterized by the Verblunsky coefficients $\hat\alpha_{2k}=0$ and $\hat\alpha_{2k+1}=\alpha_k$. They are orthonormal with respect to a measure which is is invariant for the mapping $\theta \mapsto \theta+\pi$ --i.e., symmetric under the change of sign of $e^{i\theta}$--, namely,
$$
 d\hat\mu(\theta) = \frac{1}{2}\left(d\mu(2\theta)+d\mu(2\theta-2\pi)\right),
$$
where $\mu$ is the measure on $[0,2\pi]$ which makes $\varphi_n$ orthonormal, so that $d\mu(2\theta)$ and $d\mu(2\theta-2\pi)$ are supported on $[0,\pi]$ and $[\pi,2\pi]$ respectively.

Since $d\mu(\theta)=2d\hat\mu(\theta/2)$, we can express the OPRL \eqref{eq:q} and the measure \eqref{eq:nu} for the basic DVZ transform as
\begin{equation} \label{eq:q-nu}
\begin{aligned}
 & q_n(x) =
 z^{-n} \, \frac{\hat\varphi_{2n+1}^*(z)+\hat\varphi_{2n+1}(z)}{1+z},
 \qquad
 x=z+z^{-1},
 \\[2pt]
 & d\nu(x) = (2+x) \, d\hat\mu(\arccos(x/2)),
 \qquad
 x\in[-2,2].
\end{aligned}
\end{equation}
Given a measure $\mu$ on $\T$ which is symmetric under conjugation, its Szeg\H o projection is the measure $\sigma=\Sz(\mu)$ induced on $[-2,2]$ by the mapping $z \mapsto x=z+z^{-1}$, which is explicitly given by $d\sigma(x)=2d\mu(\arccos(x/2))$. The above expression for the measure $\nu$ shows that the basic DVZ connection is a combination of three transformations: symmetrization $\mu \mapsto \hat\mu$, Szeg\H o projection $\hat\mu \mapsto \hat\sigma=\Sz(\hat\mu)$ and the Christoffel transformation $d\hat\sigma(x) \mapsto \frac{1}{2}(2+x)\,d\hat\sigma(x)$.

Let us have a closer look at the effect of these transformations on the orthogonal polynomials. The OPRL associated with the Szeg\H o projection $\sigma=\Sz(\mu)$ have the form \cite{SiOPUC,Sz}
\begin{equation} \label{eq:SzOPRL}
 p_n(x) = z^{-n}
 \frac{\varphi_{2n}(z)+\varphi_{2n}^*(z)}
 {\sqrt{2(1-\alpha_{2n-1})}},
 \qquad
 x=z+z^{-1}.
\end{equation}
On the other hand, due to the symmetry $e^{i\theta} \mapsto -e^{i\theta}$ of $\hat\mu$, its Szeg\H o projection $\hat\sigma=\Sz(\hat\mu)$ becomes an even measure, i.e. symmetric under $x \mapsto -x$, thus its OPRL have the form \cite[Chapter I]{Ch}
$$
 \hat{p}_{2n}(x) = P_n(x^2), \qquad \hat{p}_{2n+1}(x) = x\widetilde{P}_n(x^2),
$$
with $P_n$ and $\widetilde{P}_n$ polynomials with orthonormality measures $dm(x)$ and $x\,dm(x)$ supported on $\R_+$. Bearing in mind \eqref{eq:symOPUC}, an expression similar to \eqref{eq:SzOPRL} for the OPRL $\hat{p}_n$ related to the the Szeg\H o projection $\hat\sigma=\Sz(\hat\mu)$ shows that
$$
 \hat{p}_n(x) = z^{-n}
 \frac{\varphi_n(z^2)+\varphi_n^*(z^2)}
 {\sqrt{2(1-\alpha_{n-1})}},
 \qquad
 x=z+z^{-1},
$$
which coincides with \eqref{eq:DGOPRL} once we substitute $z$ by $z^{1/2}$. This shows that the DG connection follows just by combining the symmetrization and the Szeg\H o projection, thus the DG measure on $[-2,2]$ is given by $d\hat\sigma(x)=2d\hat\mu(\arccos(x/2))=d\mu(2\arccos(x/2))$. Besides,
\begin{equation} \label{eq:peven}
 \hat{p}_{2n}(x) = \hat{p}_{2n}(z+z^{-1}) = p_n(z^2+z^{-2}) = p_n(x^2-2).
\end{equation}
Therefore, $P_n(x)=p_n(x-2)$, $dm(x)=d\sigma(x-2)$ and $\widetilde{P}_n(x)=\tilde{p}_n(x-2)$, where $\tilde{p}_n$ are orthonormal with respect to $d\tilde\sigma(x)=(x+2)\,d\sigma(x)$. Hence,
\begin{equation} \label{eq:podd}
 \hat{p}_{2n+1}(x) = x \, \tilde{p}_n(x^2-2),
\end{equation}
so that
$$
 \tilde{p}_n(z^2+z^{-2}) = \frac{1}{x} \, \hat{p}_{2n+1}(x) =
 z^{-(2n+1)}
 \frac{\hat\varphi_{4n+2}(z)+\hat\varphi_{4n+2}^*(z)}
 {\sqrt{2(1-\hat\alpha_{4n+1})}(z+z^{-1})}.
$$
Combining this with \eqref{eq:symOPUC} we find that
\begin{equation} \label{eq:SzOPRL2}
 \tilde{p}_n(x) =
 z^{-n}
 \frac{\varphi_{2n+1}(z)+\varphi_{2n+1}^*(z)}{\sqrt{2(1-\alpha_{2n})}(1+z)},
 \qquad
 x=z+z^{-1}.
\end{equation}
Bearing in mind that $\hat\alpha_{2n}=0$, this identifies $q_n$, as it is given in \eqref{eq:q-nu}, with the OPRL related to the measure $\frac{1}{2}(x+2)\,d\hat\sigma(x)$, in agreement with the previous interpretation of DVZ as a composition of the symmetrization, Szeg\H o and Christoffel transformations.

Using the following notation,
$$
\begin{aligned}
 & {\cal C} \xrightarrow{\; z \mapsto z^2} \widehat{\cal C}
 & & \text{Symmetrization process between CMV matrices,}
 \\
 & {\cal C} \xrightarrow{\ \Sz \;} {\cal J}
 & & \text{Szeg\H o projection between CMV and Jacobi matrices,}
 \\
 & {\cal J} \xrightarrow{\ \wp \;} {\cal K}
 & & \parbox{260pt}{\hfil\break
 Christoffel transformation between Jacobi matrices
 which multiplies the measure by the polynomial $\wp$,}
 \\
 & {\cal C} \xrightarrow{\ \DG \ } {\cal J}
 & & \text{DG connection between CMV and Jacobi matrices,}
 \\
 & {\cal C} \xrightarrow{\ \DVZ \ } {\cal K}
 & & \text{Basic DVZ connection between CMV and Jacobi matrices,}
\end{aligned}
$$
the above results are summarized in the commutative diagram below.
$$
 \xymatrix{
	\widetilde{\cal C} \ar[d]|{\Sz} &&
 	{\cal C} \ar[ll]_{2+z+z^{-1}} \ar[rr]^{z \mapsto z^2} \ar[d]|{\Sz}
	\ar@(r,r)[rrrrd]^{\DVZ}
	\ar@(d,u)[rrd]^{\DG} &&
	\widehat{\cal C} \ar[rr]^{\frac{1}{2}(2+z+z^{-1})} \ar[d]|{\Sz} &&
	\widehat{\cal D} \ar[d]|{\Sz}
	\\
	\widetilde{\cal J} &&
	{\cal J} \ar[ll]_{2+x} \ar@{=>}[rr]^{x \mapsto x^2-2} &&
	\widehat{\cal J} \ar[rr]^{\frac{1}{2}(2+x)} &&
	{\cal K} &
	\kern-27pt = {\cal L}+{\cal M}.
 }
$$

The above diagram has more than the three alluded transformations whose combination gives the basic DVZ connection. Let us comment on them. First, every Christoffel transformation $d\sigma(x) \mapsto \wp(x)\,d\sigma(x)$ between measures on $[-2,2]$ is lifted to a similar one $d\mu(\theta) \mapsto \wp(e^{i\theta}+e^{-i\theta})\,d\mu(\theta)$ between the measures on $\T$ from which they come as Szeg\H o projections. This explains the CMV matrix $\widehat{\cal D}$ in the right upper corner: its measure $\frac{1}{2}|z+1|^2d\hat\mu(z)$ is the Christoffel transform on the unit circle (see \cite{CMMV,DaHeMa,GaMa} for this notion) whose Szeg\H o projection is the measure $\nu$ of the basic DVZ transform ${\cal K}$.

Also, the Szeg\H o projection of the symmetrization process on $\T$ yields a transformation which maps any measure $\sigma$ on $[-2,2]$ into an even measure $\hat\sigma$ on the same interval. Since $z \mapsto z^2$ reads as $x \mapsto x^2-2$ if $x=z+z^{-1}$, we conclude that
$$
 d\hat\sigma(x) = \frac{1}{2} (d\sigma(\xi(x))+d\sigma(\xi(-x))),
 \qquad
 \xi(x)=x^2-2,
$$
where $\xi$ stands for the one-to-one mapping between $[0,2]$ and $[-2,2]$ induced by $x \mapsto x^2-2$, so that $d\sigma(\xi(x))$ is a measure on $[0,2]$, while $d\sigma(\xi(-x))$ is a measure on $[-2,0]$. Nevertheless, \eqref{eq:peven} and \eqref{eq:podd} show that the OPRL $\hat{p}_n$ related to $\hat\sigma$ are not built only out of the OPRL $p_n$ for $\sigma$, but the OPRL $\tilde{p}_n$ of its Christoffel transform $d\tilde\sigma(x)=(2+x)\,d\sigma(x)$ are also required. Actually, the relations \eqref{eq:peven} and \eqref{eq:podd} between $p_n$, $\tilde{p}_n$ and $\hat{p}_n$ imply that the corresponding Jacobi matrices, ${\cal J}$, $\widetilde{\cal J}$ and $\widehat{\cal J}$, are linked by
$$
 \widehat{\cal J}^2 - 2I = {\cal J} \oplus \widetilde{\cal J},
$$
where ${\cal J}$ and $\widetilde{\cal J}$ act on even and odd indices respectively. We express this diagrammatically as
$$
 \widetilde{\cal J} \xleftarrow{\ 2+x \ } {\cal J}
 \xRightarrow{\ x \mapsto x^2-2 \ } \widehat{\cal J}
$$
the double line in the right arrow indicating that both, ${\cal J}$ and $\widetilde{\cal J}$, are involved in the construction of $\widehat{\cal J}$.

The rest of the diagram, i.e. the connection between the CMV matrices ${\cal C}$ and $\widetilde{\cal C}$ in the left upper corner, is just the result of lifting to $\T$ the Christoffel transformation relating ${\cal J}$ and $\widetilde{\cal J}$.

%%%%%%%%%%%%%%%%%%%%%%%%%%%%%%%%%%%%%%%%%%%%%%%%%%%%%%%%%%%%%%%%%%%%%%%%%%
\section{The general DVZ connection}
\label{sec:gDVZ}
%%%%%%%%%%%%%%%%%%%%%%%%%%%%%%%%%%%%%%%%%%%%%%%%%%%%%%%%%%%%%%%%%%%%%%%%%%

In this section we intend to extend the basic DVZ connection to arbitrary real linear combinations of the factors ${\cal L}$, ${\cal M}$ of a CMV matrix ${\cal C}={\cal M}{\cal L}$, i.e. a linear pencil which we denote by
$$
 {\cal K}_{\lambda_0,\lambda_1} \kern-1pt = \kern-1pt
 \lambda_0{\cal L}+\lambda_1{\cal M} \kern-1pt = \kern-2pt
 \left(
 \begin{smallmatrix}
 	\lambda_0\alpha_0+\lambda_1 &
	\kern-5pt \lambda_0\rho_0
 	\\[5pt]
 	\lambda_0\rho_0 &
	\kern-5pt -\lambda_0\alpha_0+\lambda_1 \alpha_1 &
	\kern-3pt \lambda_1\rho_1
 	\\[5pt]
 	& \kern-5pt \lambda_1\rho_1 &
	\kern-3pt \lambda_0\alpha_2-\lambda_1\alpha_1 &
    \kern-6pt \lambda_0\rho_2
 	\\[5pt]
 	& & \kern-3pt \lambda_0\rho_2 &
	\kern-6pt -\lambda_0\alpha_2+\lambda_1\alpha_3 &
	\kern9pt \lambda_1\rho_3
 	\\[-2pt]
 	& & & \kern-50pt \ddots & \kern-40pt \ddots & \ddots
 \end{smallmatrix}
 \right) \kern-1pt ,
 \kern9pt \lambda_k\in\R\setminus\{0\}.
$$
${\cal K}_{\lambda_0,\lambda_1}$ is a Jacobi matrix when $\lambda_k>0$, otherwise it is related to a true Jacobi matrix ${\cal J}_{\lambda_0,\lambda_1}$ by conjugation with a diagonal sign matrix,
$$
\begin{gathered}
 {\cal J}_{\lambda_0,\lambda_1} =
 \Pi_{\varepsilon_0,\varepsilon_1}
 {\cal K}_{\lambda_0,\lambda_1}
 \Pi_{\varepsilon_0,\varepsilon_1} =
 \left(
 \begin{smallmatrix}
    \\
 	\lambda_0\alpha_0+\lambda_1 &
	\kern-3pt |\lambda_0|\rho_0
 	\\[5pt]
 	|\lambda_0|\rho_0 &
	\kern-3pt -\lambda_0\alpha_0+\lambda_1 \alpha_1 &
	|\lambda_1|\rho_1
 	\\[5pt]
 	& \kern-3pt |\lambda_1|\rho_1 &
	\lambda_0\alpha_2-\lambda_1\alpha_1 &
    \kern-3pt |\lambda_0|\rho_2
 	\\[5pt]
 	& & |\lambda_0|\rho_2 &
	\kern-3pt -\lambda_0\alpha_2+\lambda_1\alpha_3 &
	\kern12pt |\lambda_1|\rho_3
 	\\[-2pt]
 	& & & \kern-60pt \ddots & \kern-55pt \ddots & \kern-5pt \ddots
	\\[2pt]
 \end{smallmatrix}
 \right),
 \\
 \Pi_{\lambda_0,\lambda_1} =
 \left(
 \begin{smallmatrix}
 	\\[2pt]
 	1 \\[1pt]
 	& \kern4pt \varepsilon_0 \\[2pt]
 	& & \kern-2pt \varepsilon_0\varepsilon_1 \\[2pt]
 	& & & \kern-2pt \varepsilon_1 \\[1pt]
 	& & & & \kern4pt 1 \\[1pt]
 	& & & & & \kern4pt \varepsilon_0 \\[2pt]
 	& & & & & & \kern-2pt \varepsilon_0\varepsilon_1 \\[2pt]
 	& & & & & & & \kern-2pt \varepsilon_1 \\[-5pt]
 	& & & & & & & & \ddots
 \\[2pt]
 \end{smallmatrix}
 \right),
 \qquad \varepsilon_k=\sgn(\lambda_k).
\end{gathered}
$$
As in the case of the basic DVZ connection, our aim is to find explicit relations between the measures and orthogonal polynomials associated with ${\cal C}$ and ${\cal J}_{\lambda_0,\lambda_1}$.

Prior to the discussion of these relations we will clarify the OPRL targets of this connection with OPUC. In other words, which Jacobi matrices may be expressed as ${\cal J}_{\lambda_0,\lambda_1}={\cal J}_{\lambda_0,\lambda_1}((\alpha_n)_{n\ge0})$ for some sequence $(\alpha_n)_{n\ge0}$ of real Verblunsky coefficients? This question is answered by the following proposition.

\begin{prop} \label{prop:J-DVZ}
A Jacobi matrix ${\cal J}$ given by \eqref{eq:Jac} has the form ${\cal J}_{\lambda_0,\lambda_1}((\alpha_n)_{n\ge0})$ for some $\lambda_k\in\R\setminus\{0\}$ and a sequence $(\alpha_n)_{n\ge0}$ in $(-1,1)$ iff the complex numbers $z_n=\sum_{j=0}^nb_j+ia_n$ satisfy the following conditions for $n\ge0$:
\begin{itemize}
\item[(i)] All the points $z_n$ with even index $n$ lie in a single circumference $C_0$ centered at a point in $\R\setminus\{0\}$.
\item[(ii)] All the points $z_n$ with odd index $n$ lie in the circumference $C_1$ with the same center as $C_0$ and passing through the origin.
\end{itemize}
If this is the case, $\lambda_1$ is the common center of $C_0$ and $C_1$, $|\lambda_0|$ is the radius of $C_0$ and, for any sign of $\lambda_0$,
\begin{equation} \label{eq:alpha}
 \alpha_n = \frac{1}{\lambda_{\epsilon(n)}}
 \left(\sum_{j=0}^n b_j -\lambda_1\right),
 \qquad \epsilon(n) = n \kern-8pt \mod 2,
 \qquad n\geq 0,
\end{equation}
so that the mapping $(\lambda_0,\lambda_1,(\alpha_n)_{n\ge0}) \mapsto {\cal J}_{\lambda_0,\lambda_1}((\alpha_n)_{n\ge0})$ is one-to-one up to the identity
\begin{equation} \label{eq:inj}
 {\cal J}_{\lambda_0,\lambda_1}((\alpha_n)_{n\ge0}) =
 {\cal J}_{-\lambda_0,\lambda_1}(((-1)^{n+1}\alpha_n)_{n\ge0}).
\end{equation}
\end{prop}

\begin{proof}
The equality ${\cal J}={\cal J}_{\lambda_0,\lambda_1}((\alpha_n)_{n\ge0})$ is equivalent to the conditions
\begin{equation} \label{eq:J-DVZ}
 \sum_{j=0}^nb_j -\lambda_1 = \lambda_{\epsilon(n)}\alpha_n,
 \qquad
 a_n = |\lambda_{\epsilon(n)}|\rho_n,
 \qquad
 n\ge0.
\end{equation}
From these conditions we find \eqref{eq:alpha} and
\begin{equation} \label{eq:circ}
 \left(\sum_{j=0}^nb_j -\lambda_1\right)^2 + a_n^2 =
 \lambda_{\epsilon(n)}^2,
 \qquad n\ge0,
\end{equation}
which proves {\it (i)}, {\it (ii)}, as well as the relations between $\lambda_k$ and the circumferences $C_k$. The equality \eqref{eq:inj} follows directly from the explicit form of ${\cal J}_{\lambda_0,\lambda_1}((\alpha_n)_{n\ge0})$.

Suppose now that {\it (i)} and {\it (ii)} are true, which means that \eqref{eq:circ} holds for some $\lambda_k\in\R\setminus\{0\}$, where the sign of $\lambda_0$ may be arbitrarily chosen. Introducing $\rho_n=a_n/|\lambda_{\epsilon(n)}|$ and defining $\alpha_n$ by \eqref{eq:alpha}, the relation \eqref{eq:circ} becomes $\alpha_n^2+\rho_n^2=1$. We conclude that $(\alpha_n)_{n\ge0}$ is a sequence in $(-1,1)$ such that $\rho_n=\sqrt{1-a_n^2}$ and \eqref{eq:J-DVZ} is satisfied, i.e. ${\cal J}={\cal J}_{\lambda_0,\lambda_1}((\alpha_n)_{n\ge0})$. Regarding this equality, according to this discussion, the only freedom in $\lambda_k$ and $\alpha_n$ is the choice of the sign of $\lambda_0$, whose alteration only changes the sign of the parameters $\alpha_n$ with even index $n$. This proves that the map $(\lambda_0,\lambda_1,(\alpha_n)_{n\ge0}) \mapsto {\cal J}_{\lambda_0,\lambda_1}((\alpha_n)_{n\ge0})$ fails to be one-to-one only due to \eqref{eq:inj}.
\end{proof}

The previous proposition states that the Jacobi matrices arising from the present general version of DVZ are generated by selecting a couple of concentric circumferences --$C_0$ and $C_1$-- with center on $\R\setminus\{0\}$, one of them --$C_1$-- passing though the origin. The choice of a sequence $z_n\in C_{\epsilon(n)}$ with positive imaginary part determines the corresponding Jacobi parameters $a_n=\im z_n$ and $b_n=\re(z_n-z_{n-1})$.

Due to the freedom in the sign of $\lambda_0$, we can suppose without loss of generality that $\lambda_0>0$. Also, rewriting
$$
 {\cal J}_{\lambda_0,\lambda_1} = \lambda_0 {\cal J}_\lambda,
 \qquad\quad {\cal J}_\lambda = {\cal J}_{1,\lambda},
 \qquad\quad \lambda=\lambda_1/\lambda_0\in\R\setminus\{0\},
$$
we find that the OPRL $p^{(\lambda_0,\lambda_1)}_n$ and the measure $\nu_{\lambda_0,\lambda_1}$ related to ${\cal J}_{\lambda_0,\lambda_1}$ follow by a simple dilation of the OPRL $p^{(\lambda)}_n$ and the measure $\nu_\lambda$ associated with ${\cal J}_\lambda$,
$$
 p^{(\lambda_0,\lambda_1)}_n(x) = p^{(\lambda)}_n(x/\lambda_0),
 \qquad\quad
 d\nu_{\lambda_0,\lambda_1}(x) = d\nu_\lambda(x/\lambda_0).
$$
Besides,
$$
 {\cal J}_\lambda =
 \begin{cases}
 	{\cal K}_\lambda, & \lambda>0,
 	\\
 	\Pi{\cal K}_\lambda\Pi, & \lambda<0,
 \end{cases}
 \qquad\quad
 {\cal K}_\lambda = {\cal K}_{1,\lambda} = {\cal L}+\lambda{\cal M},
$$
where $\Pi$ is the diagonal sign matrix given in \eqref{eq:Pi}. Thus, the OPRL sequence $p^{(\lambda)}=(p^{(\lambda)}_0,p^{(\lambda)}_0,\dots)^t$ given by ${\cal J}_\lambda p^{(\lambda)}(x)=xp^{(\lambda)}(x)$, $p^{(\lambda )}_0=1$, is related by $p^{(\lambda)}=\Pi q^{(\lambda)}$ with the solutions of
\begin{equation} \label{eq:Kq}
 {\cal K}_\lambda q^{(\lambda)}(x) = x q^{(\lambda)}(x),
 \qquad\quad
 q^{(\lambda)}=(q^{(\lambda)}_0,q^{(\lambda)}_1,\dots)^t,
 \qquad\quad q_0^{(\lambda)}=1.
\end{equation}
Hence, the polynomials $q_n^{(\lambda)}$ and $p_n^{(\lambda)}$ differ in at most a sign, so $q^{(\lambda)}$ is also an OPRL sequence with respect to the measure $\nu_\lambda$. Due to these reasons, in what follows we will consider only linear combinations of CMV factors with the form ${\cal K}_\lambda$ for an arbitrary $\lambda\in\R\setminus\{0\}$ , referring to the solutions $q_n^{(\lambda)}$ of \eqref{eq:Kq} as the corresponding OPRL. We will refer to the problem of finding the relations between the orthogonal polynomials and measures related to a CMV matrix ${\cal C}={\cal M}{\cal L}$ and the tridiagonal matrix ${\cal K}_\lambda={\cal L}+\lambda{\cal M}$ as the (general) DVZ connection between OPUC and OPRL.

This general version of DVZ was already considered in \cite{DVZ}. Nevertheless, unlike the case of the basic DVZ connection, \cite{DVZ} does not provide any closed formula for the relations between the orthogonal polynomials and measures associated with a CMV matrix ${\cal C}={\cal M}{\cal L}$ and its (general) DVZ transform,
\begin{equation} \label{eq:Kl}
 {\cal K}_\lambda = {\cal L}+\lambda{\cal M} =
 \begin{pmatrix}
 	\alpha_0+\lambda &
	\kern-3pt \rho_0
 	\\[2pt]
 	\rho_0 &
	\kern-3pt \lambda \alpha_1-\alpha_0 &
	\kern-3pt \lambda\rho_1
 	\\[2pt]
 	& \kern-3pt \lambda\rho_1 &
	\kern-3pt \alpha_2-\lambda\alpha_1 &
    \kern-3pt \rho_2
 	\\[2pt]
 	& & \kern-3pt \rho_2 &
	\kern-3pt \lambda\alpha_3-\alpha_2 &
	\kern12pt \lambda\rho_3
 	\\
 	& & & \kern-67pt \ddots & \kern-47pt \ddots & \ddots
 \end{pmatrix}.
\end{equation}
This will change with the present approach, whose simplicity leads to completely explicit expressions for such general DVZ relations between orthogonal polynomials and measures.

%%%%%%%%%%%%%%%%%%%%%%%%%%%%%%%%%%%%%%%%%%%%%%%%%%%%%%
\subsection {OPRL for the general DVZ connection}
\label{ssec:OP-gDVZ}
%%%%%%%%%%%%%%%%%%%%%%%%%%%%%%%%%%%%%%%%%%%%%%%%%%%%%%

We are looking for the solutions of the formal eigenvalue equation
\begin{equation} \label{eq:Kl-ql}
 {\cal K}_\lambda q^{(\lambda)}(x) = x q^{(\lambda)}(x),
 \qquad\quad
 q^{(\lambda)} = (q^{(\lambda)}_0,q^{(\lambda)}_1,q^{(\lambda)}_2,\dots)^t,
 \qquad\quad q^{(\lambda)}_0=1,
\end{equation}
which yield an OPRL sequence $q_n^{(\lambda)}$ with positive leading coefficients for $\lambda>0$, and having a sign given by the diagonal entries of $\Pi$ in \eqref{eq:Pi} if $\lambda<0$. The same method used previously for $\lambda=\pm1$ provides the relation between the OPRL $q_n^{(\lambda)}$ and the OLPUC $\chi_n$ of ${\cal C}$ for any value of $\lambda\in\R\setminus\{0\}$. Since
$$
 {\cal K}_\lambda\chi(z)=(z+\lambda)\chi_*(z),
 \qquad\quad
 {\cal K}_\lambda\chi_*(z)=(z^{-1}+\lambda)\chi(z),
$$
${\cal K}_\lambda$ has formal eigenvectors $u(z)\chi(z)+v(z)\chi_*(z)$ determined by
$$
 \begin{pmatrix}
 0 & z^{-1}+\lambda
 \\
 z+\lambda & 0
 \end{pmatrix}
 \begin{pmatrix} u(z) \\ v(z) \end{pmatrix} = x
 \begin{pmatrix} u(z) \\ v(z) \end{pmatrix}.
$$
The eigenvalues $x$ are the solutions of
\begin{equation} \label{eq:x-l}
 x^2 = (z+\lambda)(z^{-1}+\lambda)
 = 1+\lambda^2 +\lambda(z+z^{-1}),
\end{equation}
the corresponding eigenvectors being spanned by
\begin{equation} \label{eq:Xl}
 X(z) = \chi(z) + \frac{z+\lambda}{x} \chi_*(z).
\end{equation}
This suggests the identification
\begin{equation} \label{eq:ql}
 q_n^{(\lambda)}(x) = \frac{X_n(z)}{X_0(z)}
 = \frac{x\chi_n(z)+(z+\lambda)\chi_{n*}(z)}{x+z+\lambda},
 \qquad
 z=z(x),
\end{equation}
where $z=z(x)$ is a complex mapping satisfying \eqref{eq:x-l}. The validity of this result only needs to check that the above expression yields a well defined function of $x$ because then \eqref{eq:Kl-ql} follows from the eigenvector property of $X(z)$. However, \eqref{eq:x-l} determines $x$ as function of $z$ only up to a sign which could depend arbitrarily on $z$. Also, although $z \mapsto 1+\lambda^2+\lambda(z+z^{-1})$ maps $\C\setminus\{0\}$ onto $\C$, there are in general two values of $z$ giving the same value of $x^2$ which are inverse of each other. Therefore, all that can be said is that \eqref{eq:x-l} establishes a one-to-one correspondence between subsets $\{z,z^{-1}\}\subset\C\setminus\{0\}$ and subsets $\{x,-x\}\subset\C$. As a consequence, there are infinitely many functions $z=z(x)$ well defined on the whole complex plane and satisfying \eqref{eq:x-l}, so that \eqref{eq:ql} becomes a true function of $x$ when substituting $z$ by $z(x)$. The fact that, for any of these choices, $q_n^{(\lambda)}(x)$ satisfies \eqref{eq:Kl-ql}, not only identifies these functions as the OPRL related to ${\cal K}_\lambda$, but consequently shows that $q_n^{(\lambda)}(x)$ do not depend on the particular choice of $z(x)$. A more direct proof of this result is given by the following theorem, which is of practical interest since it provides an alternative expression of $q_n^{(\lambda)}(x)$ which turns out to be more useful to identify its dependence on $x$ in particular situations.

\begin{thm} \label{thm:ql}
Let $\chi_n$ be the OLPUC of a real CMV matrix with $\Theta$-factorization ${\cal C}={\cal M}{\cal L}$. Then, the function $q_n^{(\lambda)}(x)$ given in \eqref{eq:ql} does not depend on the map $z=z(x)$ satisfying \eqref{eq:x-l}, and is a real polynomial of degree $n$ with the form
$$
 q_n^{(\lambda)}(x) =
 Q_n(t_\lambda(x)) + (x-\lambda)\widetilde{Q}_n(t_\lambda(x)),
 \qquad
 t_\lambda(x) = \lambda^{-1} x^2 -(\lambda+\lambda^{-1}),
$$
where $Q_n$ and $\widetilde{Q}_n$ are real polynomials independent of $\lambda$ with degree
$$
 \begin{aligned}
 	& \deg Q_n = \textstyle \frac{n}{2}, & & \text{even} \; n,
	\\
	& \deg Q_n \le \textstyle \frac{n-1}{2} & & \text{odd} \; n,
 \end{aligned}
 \qquad\quad
 \deg \widetilde{Q}_n = \left[\textstyle\frac{n-1}{2}\right],
$$
such that
$$
 Q_n(z+z^{-1}) = \frac{1}{z-z^{-1}}
 \begin{vmatrix}
 	z & \kern-3pt z^{-1}
	\\
 	\chi_n(z) & \kern-3pt \chi_{n*}(z)
 \end{vmatrix},
 \qquad
 \widetilde{Q}_n(z+z^{-1}) = \frac{1}{z-z^{-1}}
 \begin{vmatrix}
 	\chi_n(z) & \kern-3pt \chi_{n*}(z)
	\\
	1 & \kern-3pt 1
 \end{vmatrix}.
$$
If $\chi_n(z) = \sum_j c_j^{(n)}z^j$, then
$$
 Q_n(t) = \sum_{j\ge0} (c_{-j}^{(n)}-c_{j+2}^{(n)}) U_j(t),
 \qquad
 \widetilde{Q}_n(t) = \sum_{k\ge0} (c_{j+1}^{(n)}-c_{-j-1}^{(n)}) U_j(t),
$$
where $U_n$ stand for the second kind Chebyshev polynomials on $[-2,2]$, determined by the recurrence relation
\begin{equation} \label{eq:cheb}
 U_{-1}(t)=0, \qquad
 U_0(t)=1, \qquad
 U_n(t) = tU_{n-1}(t)-U_{n-2}(t) \; \text{ if } \; n\ge1.
\end{equation}
The polynomials $q_n^{(\lambda)}(x)$ satisfy \eqref{eq:Kl-ql}, i.e. they are the OPRL related to the tridiagonal matrix ${\cal K}_\lambda={\cal L}+\lambda{\cal M}$.
\end{thm}

\begin{proof}
Given any map $z=z(x)$ satisfying \eqref{eq:x-l}, from \eqref{eq:ql} we find that
$$
\begin{aligned}
 q_n^{(\lambda)}(x)
 & = \frac{(x\chi_n(z)+(z+\lambda)\chi_{n*}(z)) (x-z-\lambda)}
	{x^2-(z+\lambda)^2}
 \\
 & = \frac{x^2\chi_n(z) - x(z+\lambda)\chi_n(z)
 		+ x(z+\lambda)\chi_{n*}(z) - (z+\lambda)^2\chi_{n*}(z)}
 	{(z+\lambda)(z^{-1}-z)}
 \\
 & = \frac{(z+\lambda)\chi_{n*}(z) -  (z^{-1}+\lambda)\chi_n(z)
 		+ x\chi_n(z) - x\chi_{n*}(z)}
 	{z-z^{-1}}
	\\
 & = \frac{1}{z-z^{-1}}
 	\left(
 	\begin{vmatrix}
 		z & \kern-3pt z^{-1}
		\\
 		\chi_n(z) & \kern-3pt \chi_{n*}(z)
 	\end{vmatrix}
 	+ (x-\lambda)
 	\begin{vmatrix}
 		\chi_n(z) & \kern-3pt \chi_{n*}(z)
		\\
		1 & \kern-3pt 1
 	\end{vmatrix}
	\right).
\end{aligned}
$$
Since the columns of the above determinants are related by the substar operation, they must be real linear combinations of $z^j-z^{-j}$. Bearing in mind that
\begin{equation} \label{eq:U-z}
 U_j(z+z^{-1}) = \frac{z^{j+1}-z^{-j-1}}{z-z^{-1}},
\end{equation}
we conclude that both determinants, when divided by $z-z^{-1}$, become real linear combinations of $U_j(z+z^{-1})$, which thus have the form $Q_n(z+z^{-1})$ and $\widetilde{Q}_n(z+z^{-1})$ for some real polynomials $Q_n$ and $\widetilde{Q}_n$. The relations \eqref{eq:OLPUC-OPUC} lead to
$$
\begin{aligned}
 & \deg Q_{2k} = k, & \qquad & \deg Q_{2k+1} \le k,
 \\
 & \deg \widetilde{Q}_{2k} = k-1, & & \deg \widetilde{Q}_{2k+1} = k.
\end{aligned}
$$
This implies that $\deg q_n^{(\lambda)}=n$ because $z+z^{-1}=\lambda^{-1}x^2-(\lambda+\lambda^{-1})$ due to \eqref{eq:x-l}.

Introducing the full expansion $\chi_n(z) = \sum_k c_j^{(n)}z^j$ into the determinants defining $Q_n$ and $\widetilde{Q}_n$ we obtain
$$
 Q_n(t) = - \sum_j c_j^{(n)} U_{j-2}(t),
 \qquad
 \widetilde{Q}_n(t) = \sum_j c_j^{(n)} U_{j-1}(t),
$$
where we assume \eqref{eq:U-z} extended to every $j\in\Z$. The relations given in the theorem follow by noticing that $U_{-j}=-U_{j-2}$.

Finally, the relation $q^{(\lambda)}(x) = X(z)/X_0(z)$ shows that it is a formal eigenvector of ${\cal K}_\lambda$ with eigenvalue $x$ and such that $q_0^{(\lambda)}=1$, i.e. it satisfies \eqref{eq:Kl-ql}.
\end{proof}

%%%%%%%%%%%%%%%%%%%%%%%%%%%%%%%%%%%%%%%%%%%%%%%%%%%%%%
\subsection {Orthogonality measure for the general DVZ connection}
\label{ssec:M-gDVZ}
%%%%%%%%%%%%%%%%%%%%%%%%%%%%%%%%%%%%%%%%%%%%%%%%%%%%%%

To obtain the orthogonality measure $\nu_\lambda$ of the OPRL $q_n^{(\lambda)}$   we will introduce the measures $\mu_+$ and $\mu_-$ induced by $\mu$ on the upper and lower arcs of the unit circle, $\T_+=\{e^{i\theta}:\theta\in[0,\pi]\}$ and $\T_-=\{e^{i\theta}:\theta\in[-\pi,0]\}$, with any eventual mass point of $\mu$ at $\pm1$ equally distributed between $\mu_+$ and $\mu_-$. Then, we can rewrite any integral with respect to $\mu$ as
$$
 \int_\T f(z) \, d\mu(z)
 = \int_{\T_+} f(z) \, d\mu_+(z)
 + \int_{\T_-} f(z) \, d\mu_-(z).
$$
Besides, the symmetry of $\mu$ under conjugation means that $d\mu_+(z)=d\mu_-(z^{-1})$, so that
$$
 \int_\T f(z) \, d\mu(z)
 = \int_{\T_+} (f(z)+f(z^{-1})) \, d\mu_+(z).
$$

The orthogonality measure $\nu_\lambda$ will arise from the orthogonality of the components of the eigenvector \eqref{eq:Xl} with respect to $\mu$, a result given in the next proposition. To make it more precise and prove this orthogonality we define the following map on $\T$
\begin{equation} \label{eq:x-t}
 x(e^{i\theta}) = \sqrt{1+\lambda^2+2\lambda\cos\theta},
\end{equation}
which satisfies \eqref{eq:x-l} and $x(z^{-1})=x(z)$. When restricted to $\T_\pm$, \eqref{eq:x-t} provides a homeomorphism between $\T_\pm$ and the interval $[|1-|\lambda||,1+|\lambda|]$, with an inverse mapping given by $z(x)=e^{\pm i\theta(x)}$, where
\begin{equation} \label{eq:t-x}
 \theta(x) = \arccos\frac{t_\lambda(x)}{2},
 \qquad
 t_\lambda(x) = \lambda^{-1}x^2-(\lambda+\lambda^{-1}).
\end{equation}
Now we can formulate the precise meaning of the orthogonality property for the eigenvector \eqref{eq:Xl}.

\begin{prop} \label{prop:X-ort}
Let $\chi_n$ be the OLPUC associated with a measure $\mu$ on $\T$ symmetric under conjugation. Consider the measures $\mu_\pm$ induced by $\mu$ on the upper and lower arcs $\T_\pm$, with any mass point of $\mu$ at $\pm1$ equally distributed between $\mu_\pm$. Then, if $x(z)$ is given by \eqref{eq:x-t} for $z\in\T$,
$$
\begin{gathered}
 \int_{\T_+} X^+(z)X^+(z)^\dag \, d\mu_+(z) +
 \int_{\T_-} X^-(z)X^-(z)^\dag \, d\mu_-(z) = 2I,
 \\
 X^\pm(z) = \chi(z) \pm \frac{z+\lambda}{x(z)} \chi_*(z).
\end{gathered}
$$

\end{prop}

\begin{proof}
Since $x(z)$ satisfies \eqref{eq:x-l}, the above sum of integrals can be rewritten as
$$
\begin{aligned}
 & \int_\T (\chi(z)\chi(z)^\dag + \chi_*(z)\chi_*(z)^\dag) \, d\mu(z)
 \\
 & + \int_{\T_+}
 \left(\frac{z+\lambda}{x(z)} \chi_*(z)\chi(z)^\dag +
 \frac{z^{-1}+\lambda}{x(z)} \chi(z)\chi_*(z)^\dag\right) d\mu_+(z)
 \\
 & - \int_{\T_-}
 \left(\frac{z+\lambda}{x(z)} \chi_*(z)\chi(z)^\dag +
 \frac{z^{-1}+\lambda}{x(z)} \chi(z)\chi_*(z)^\dag\right) d\mu_-(z)
\end{aligned}
$$
The OLPUC $\chi_n$ have real coefficients due to the symmetry of $\mu$ under conjugation. As a consequence, the integrands of the last two summands are invariant under conjugation of $z$. Since this operation exchanges $\mu_-$ and $\mu_+$, the last two terms cancel each other. On the other hand, the orthogonality of $\chi_n$ and $\chi_{n*}$ with respect to $\mu$ implies that the first integral is $2I$.
\end{proof}

The orthogonality measure of $q_n^{(\lambda)}$ arises as a consequence of the previous result. Except for the cases $\lambda=\pm1$, it is a measure supported on two disjoint symmetric intervals.

\begin{thm} \label{thm:nul}
If $\mu$ is a measure on $\T$ which is symmetric under conjugation, then the OPRL $q_n^{(\lambda)}(x)$ of the corresponding general DVZ transform ${\cal K}_\lambda$ are orthonormal with respect to the measure
$$
\begin{gathered}
 d\nu_\lambda(x) =
 \frac{(x+\lambda)^2-1}{2\lambda x} \,
 (d\mu_+(e^{i\theta_+(x)}) + d\mu_-(e^{-i\theta_-(x)})),
 \qquad
 x \in E_\lambda = E_\lambda^+ \cup E_\lambda^-,
 \\[2pt]
 E_\lambda^+ = [|1-|\lambda||,1+|\lambda|],
 \qquad
 E_\lambda^-=-E_\lambda^+,
\end{gathered}
$$
where $\theta_\pm(x)$ is the homeomorphism between $E_\lambda^\pm$ and $[0,\pi]$ given by the restriction of $\theta(x)$ in \eqref{eq:t-x} to $E_\lambda^\pm$, and $\mu_\pm$ is the measure induced by $\mu$ on $\T_\pm$ with the mass points of $\mu$ at $\pm1$ equally distributed between $\mu_\pm$.
\end{thm}

\begin{proof}
With the notation of Proposition~\ref{prop:X-ort}, we have
$$
 X^\pm(e^{\pm i\theta(x)}) = q_n^{(\lambda)}(\pm x) \, X^\pm_0(e^{\pm i\theta(x)}).
$$
Therefore, using the homeomorphisms $z=e^{\pm i\theta_+(x)}$ between $E_\lambda^+$ and $\T_\pm$, the result of such a proposition may be rewritten as
\begin{equation} \label{eq:int-ql}
\begin{aligned}
 & \int_{E_\lambda^+} q_n^{(\lambda)}(x) \, q_n^{(\lambda)}(x)^\dag \,
 |X^+_0(e^{i\theta(x)})|^2 \, d\mu_+(e^{i\theta_+(x)})
 \\
 & \kern50pt + \int_{E_\lambda^+} \, q_n^{(\lambda)}(-x)q_n^{(\lambda)}(-x)^\dag \,
 |X^-_0(e^{-i\theta(x)})|^2 \, d\mu_-(e^{-i\theta_+(x)}) = 2I.
\end{aligned}
\end{equation}
Besides,
$$
 |X^\pm_0(e^{\pm i\theta(x)})|^2 =
 \left|1\pm\frac{e^{\pm i\theta(x)}+\lambda}{x}\right|^2 =
 \left(1\pm\frac{z+\lambda}{x}\right) \left(1\pm\frac{z^{-1}+\lambda}{x}\right),
 \quad
 z = e^{\pm i\theta(x)}.
$$
Bearing in mind that $z=z(x)$ satisfies \eqref{eq:x-l}, we find that
$$
 |X^\pm_0(e^{i\theta(x)})|^2
 = 2 \pm \frac{z+z^{-1}+2\lambda}{x}
 = 2 \pm \frac{x^2-1+\lambda^2}{\lambda x}
 = \pm \frac{(x\pm\lambda)^2-1}{\lambda x}.
$$
The theorem follows by inserting this equality into \eqref{eq:int-ql} and performing in the second integral of \eqref{eq:int-ql} the change of variables $x \mapsto -x$, taking into account that it maps $E_\lambda^+$ into $E_\lambda^-$ and $\theta_-(x)=\theta_+(-x)$.
\end{proof}

We should highlight that, in the expression for $\nu_\lambda$ given by the previous theorem, $d\mu_+(e^{i\theta_+(x)})$ is a measure on $E_\lambda^+$ while $d\mu_-(e^{-i\theta_-(x)})$ is a measure on $E_\lambda^-$, so that the sum of both is even, i.e. symmetric under the reflection $x \mapsto -x$. Nevertheless, the measure $\nu_\lambda$ is not even due to the additional factor
\begin{equation} \label{eq:factor}
\frac{(x+\lambda)^2-1}{\lambda x},
\end{equation}
which, by the way, is non-negative on $E_\lambda$.

Bearing in mind that $d\mu_+(z)=d\mu_-(z^{-1})$, we can also express the measure $\nu_\lambda$ using only the measure $\mu_+$ on the upper arc $\T_+$, i.e.
\begin{equation} \label{eq:nu+}
 d\nu_\lambda(x) =
 \frac{(x+\lambda)^2-1}{2\lambda x} \,
 (d\mu_+(e^{i\theta_+(x)}) + d\mu_+(e^{i\theta_-(x)})),
 \qquad
 x \in E_\lambda.
\end{equation}
Since $\theta_\pm(x)=\theta(x)$ for $x\in E_\lambda^\pm$, it is tempting to rewrite the above as
\begin{equation} \label{eq:nu+-red}
 d\nu_\lambda(x) =
 \frac{(x+\lambda)^2-1}{2\lambda x} \,
 d\mu_+(e^{i\theta(x)}),
 \qquad
 x \in E_\lambda.
\end{equation}
However, \eqref{eq:t-x} is not one-to-one between $E_\lambda$ and $\T_+$ because $\theta(-x)=\theta(x)$, thus \eqref{eq:nu+-red} should be considered simply as a symbolic representation of \eqref{eq:nu+}.

The results of Sect.~\ref{sec:DVZ} which express $\nu_{\pm1}$ directly in terms of $\mu$ may be recovered from Theorem~\ref{thm:nul}. When $\lambda=\pm1$ the factor \eqref{eq:factor} becomes $2\pm x$, the pair of intervals constituting $E_\lambda$ reduce to the single interval $[-2,2]$, and $\theta(x) = \arccos(\pm x^2/2 \mp 1)$ so that
$$
\begin{aligned}
 & 2\cos\frac{\theta(x)}{2} = \sqrt{2(1+\cos\theta(x))} = |x|,
 & \qquad & \lambda=1,
 \\[2pt]
 & 2\sin\frac{\theta(x)}{2} = \sqrt{2(1-\cos\theta(x))} = |x|,
 & & \lambda=-1.
\end{aligned}
$$
Then, $\theta(x)$ makes sense as a one-to-one map between $[-2,2]$ and $[0,2\pi]$  $(\lambda=1$) or $[-\pi,\pi]$ $(\lambda=-1$) just by redefining
$$
\begin{aligned}
 & \theta(x) \mapsto 2\pi-\theta(x),
 & \qquad & x\in[-2,0), & \qquad & \lambda=1,
 \\
 & \theta(x) \mapsto -\theta(x),
 & & x\in[-2,0), & & \lambda=-1,
\end{aligned}
$$
which yields the transformations \eqref{eq:th} and \eqref{eq:th-} for $\lambda=1$ and $\lambda=-1$ respectively.

A representation of $\nu_\lambda$ directly in terms of $\mu$ is also possible for $\lambda\ne\pm1$. One can think on rewriting the expression given in Theorem~\ref{thm:nul} as
\begin{equation} \label{eq:nu-mu}
 d\nu_\lambda(x) =
 \frac{(x+\lambda)^2-1}{2\lambda x} \,
 d\mu(z(x)),
 \qquad
 x \in E_\lambda,
\end{equation}
with
$$
 z(x) =
 \begin{cases}
 e^{i\theta(x)}, & x \in E_\lambda^+,
 \\
 e^{-i\theta(x)}, & x \in E_\lambda^-.
 \end{cases}
$$
Although $z(x)$ fails to be one-to-one between $E_\lambda$ and $\T$ because the four edges of $E_\lambda$ are mapped into $\pm1$, this difficulty is overcome due to the factor $(x+\lambda)^2-1$ which cancels any possible mass point at the two edges $\pm1-\lambda$, making unnecessary to cover them with the mapping $z(x)$. Therefore, \eqref{eq:nu-mu} makes perfect sense considering $z(x)$ as a one-to-one map between $E_\lambda\setminus\{1-\lambda,-1-\lambda\}$ and $\T$.

The connection between $\nu_\lambda$ and $\mu$ can be made more explicit for the absolutely continuous and pure point components. Suppose that a measure $\mu$ on  $\T$ has the form
$$
 d\mu(e^{i\theta}) =
 w(\theta) \, d\theta +
 \sum_k m_k \, (\delta(\theta-\theta_k)+\delta(\theta+\theta_k)) \, d\theta,
 \qquad \theta_k\in[0,\pi],
$$
where $w(-\theta)=w(\theta)$ and the mass points appear in symmetric pairs due to the invariance of $\mu$ under conjugation. The above expression assumes for $\eta=0,\pi$ the artificial splitting of any mass point $m\,\delta(\theta-\eta)\,d\theta$ as $(m/2)(\delta(\theta-\eta)+\delta(\theta+\eta))\,d\theta$, something which allows us to identify
$$
 d\mu_\pm(e^{i\theta}) =
 w(\theta) \, d\theta +
 \sum_k m_k \, \delta(\theta\mp\theta_k) \, d\theta.
$$
Then, according to Theorem~\ref{thm:nul},
$$
 d\nu_\lambda(x) = \frac{(x+\lambda)^2-1}{2\lambda x}
 \left( w(\theta(x)) \left|\frac{d\theta}{dx}\right| \, dx +
 \sum_k m_k \, (\delta(x-x_k)+\delta(x+x_k)) \, dx \right),
$$
where $x_k=x(\theta_k)=\sqrt{1+\lambda^2+2\lambda\cos\theta_k}$.

It only remains to obtain the Jacobian of the transformation $\theta(x)$ given in \eqref{eq:t-x}, which is
$$
 \left|\frac{d\theta}{dx}\right| =
 \left|\frac{x}{\lambda\sin\theta(x)}\right|.
$$
On the other hand,
$$
\begin{aligned}
 \sin^2\theta(x) = 1-\frac{t_\lambda^2(x)}{4}
 & = \frac{((1+\lambda)^2-x^2)(x^2-(1-\lambda)^2)}{4\lambda^2}
 \\
 & = \frac{((x+\lambda)^2-1)(1-(x-\lambda)^2)}{4\lambda^2}.
\end{aligned}
$$
Therefore,
$$
 \left|\frac{d\theta}{dx}\right| =
 \frac{2|x|}{\sqrt{((x+\lambda)^2-1)(1-(x-\lambda)^2)}}.
$$

Finally, we conclude that
\begin{equation} \label{eq:nul-acpp}
\begin{aligned}
 d\nu_\lambda(x)
 & = \frac{1}{|\lambda|}
 \sqrt{\frac{(x+\lambda)^2-1}{1-(x-\lambda)^2}} \, w(\theta(x)) \, dx
 \\[3pt]
 & + \sum_k \frac{m_k}{2\lambda x_k}
 \left[ ((x_k+\lambda)^2-1)\,\delta(x-x_k)
 + (1-(x_k-\lambda)^2)\,\delta(x+x_k) \right] dx.
\end{aligned}
\end{equation}
The explicitness of this expression for $\nu_\lambda$ makes it of particular interest for the analysis of specific examples.

Among the possible mass points of $\mu$, those located at $\pm1$ deserve special attention since they are the only ones which are artificially splitted to obtain $\nu_\lambda$. If $\mu$ has a mass $m$ at $\pm1$, then $\mu_\pm$ have a mass $m/2$ at the same point which, according to the previous results, leads to the following masses for $\nu_\lambda$ at the edges of $E_\lambda$:
\begin{equation} \label{eq:mass}
\begin{aligned}
 & \mu(\{1\}) = m \kern5pt \Rightarrow \kern3pt
 \begin{cases}
 	\nu_\lambda(\{1+\lambda\}) = m,
 	\\
 	\nu_\lambda(\{-1-\lambda\}) = 0, \quad \lambda\ne-1,
 \end{cases}
 \\
 & \mu(\{-1\}) = m \kern5pt \Rightarrow \kern3pt
 \begin{cases}
  	\nu_\lambda(\{1-\lambda\}) = 0, \quad \lambda\ne1,
	\\
	\nu_\lambda(\{\lambda-1\}) = m.
 \end{cases}
\end{aligned}
\end{equation}
In agreement with our previous observation, the edges $\pm1-\lambda\in E_\lambda$ are free of mass points for the measure $\nu_\lambda$ whenever $\lambda\ne\pm1$. In the case $\lambda=\pm1$, the only point of $E_{\pm1}=[-2,2]$ which cannot support a mass point of $\nu_{\pm1}$ is $\mp2$.

%%%%%%%%%%%%%%%%%%%%%%%%%%%%%%%%%%%%%%%%%%%%%%%%%%%%%%%%%%%%%%%%%%%%%%%%%%
\section{Examples}
\label{sec:Ex}
%%%%%%%%%%%%%%%%%%%%%%%%%%%%%%%%%%%%%%%%%%%%%%%%%%%%%%%%%%%%%%%%%%%%%%%%%%

In this section we will apply the general DVZ connection to some well known families of OPUC --whose details can be found for instance in \cite[Sect.~1.6]{SiOPUC}--, which will provide new explicit examples of OPRL. Theorem~\ref{thm:ql} yields a representation of these OPRL in terms of the Chebyshev polynomials $U_n$ on $[-2,2]$ (see \eqref{eq:cheb}) evaluated on the function $t_\lambda(x)$ defined in \eqref{eq:t-x}. For convenience, in what follows we will use the abbreviation $\hat{U}_n=U_n(t_\lambda(x))$, which is a $\lambda$-dependent polynomial of degree $2n$ in $x$.

\subsection{Bernstein-Szeg\H{o} polynomials}
\label{ssec:ex-BS}

As an introductory example, we will analyze the general DVZ connection when the OPUC are the degree one real Bernstein-Szeg\H{o} polynomials, i.e. those defined by the Verblunsky coefficients
$$
 \alpha_n = \alpha \, \delta_{n,0}, \qquad \alpha\in(-1,1).
$$
The associated OPUC,
$$
 \varphi_n(z) = \rho^{-1}(z-a)z^{n-1},
 \qquad\quad
 \rho=\sqrt{1-\alpha^2},
$$
are orthonormal with respect to the measure
$$
 d\mu(e^{i\theta}) = \frac{\rho^2}{|1-\alpha e^{i\theta}|^2} \,
 \frac{d\theta}{2\pi},
 \qquad\quad \theta\in(-\pi,\pi].
$$
This measure also makes orthonormal the corresponding OLPUC which, in view of \eqref{eq:OLPUC-OPUC}, have the form
$$
 \chi_{2k}(z) = \rho^{-1} (1-\alpha z) z^{-k},
 \qquad\quad
 \chi_{2k+1}(z) = \rho^{-1} (z-\alpha) z^k.
$$
In this case, the general DVZ connection leads to a Jacobi matrix with parameters
$$
 a_n =
 \begin{cases}
 	\rho, & n=0,
 	\\
 	1, & \text{even } n\ne0,
	\\
	\lambda, & \text{odd } n,
 \end{cases}
 \qquad
 b_n =
 \begin{cases}
 	\alpha+\lambda, & n=0,
	\\
	-\alpha, & n=1,
	\\
	0, & n\ge2.
 \end{cases}
$$
Using the results of Theorem~\ref{thm:ql} we can express the related OPRL as
$$
\begin{aligned}
 & q_{2k}^{(\lambda)}(x) = \rho^{-1}
 \left[
 \hat{U}_k - (x-\lambda+\alpha)\hat{U}_{k-1} + \alpha(x-\lambda)\hat{U}_{k-2}
 \right],
 \\
 & q_{2k+1}^{(\lambda)}(x) = \rho^{-1}
 \left[
 (x-\lambda)\hat{U}_k - (\alpha(x-\lambda)+1)\hat{U}_{k-1} + \alpha \hat{U}_{k-2}
 \right].
\end{aligned}
$$
On the other hand, if $\theta=\theta(x)$ is given by \eqref{eq:t-x}, then
$$
 |1-\alpha e^{i\theta}|^2 =
 \frac{|(1+\alpha\lambda)(\lambda+\alpha) - \alpha x^2|}{|\lambda|},
$$
thus, according to \eqref{eq:nul-acpp}, the orthogonality measure of $q_n^{(\lambda)}$ reads as
$$
 d\nu_\lambda(x) =
 \frac{1}{2\pi} \sqrt{\frac{(x+\lambda)^2-1}{1-(x-\lambda)^2}} \,
 \frac{\rho^2}{|(1+\alpha\lambda)(\lambda+\alpha) - \alpha x^2|} \, dx,
 \qquad x\in E_\lambda.
$$

\subsection{Lebesgue measure with a mass point}
\label{ssec:ex-mp}

Let us consider now the  probability measure obtained by inserting a mass point at $z=1$ into the Lebesgue measure on the unit circle, i.e.
$$
 d\mu(e^{i\theta}) = (1-m)\,\frac{d\theta}{2\pi} + m\,\delta(\theta)\,d\theta,
 \qquad
 \theta\in(-\pi,\pi],
 \qquad
 m\in(0,1).
$$
The related OPUC are known to be
$$
 \varphi_n(z) = {\textstyle
 \kappa_n \left(z^n - \alpha_{n-1}\sum_{j=0}^{n-1}z^j\right),}
 \qquad
 \alpha_n = \frac{1}{n+m^{-1}},
$$
the parameters $\alpha_n$ being the Verblunsky coefficients, so that
$$
 \rho_n = \sqrt{1-\alpha_n^2} =
 \frac{\alpha_n}{\sqrt{\alpha_{n-1}\alpha_{n+1}}},
 \qquad
 \kappa_n = \prod_{j=0}^{n-1} \rho_k^{-1} =
 \sqrt{\frac{\alpha_{-1}\alpha_n}{\alpha_0\alpha_{n-1}}} =
 \sqrt{\frac{\alpha_n}{(1-m)\alpha_{n-1}}}.
$$
Using \eqref{eq:OLPUC-OPUC} we find the following expressions for the related OLPUC,
$$
 \chi_n(z) =
 \begin{cases}
 	\kappa_n \left(z^{-k}-\alpha_{n-1}\sum_{j=-k+1}^kz^j\right), & n=2k,
	\\[7pt]
	\kappa_n \left(z^{k+1}-\alpha_{n-1}\sum_{j=-k}^kz^j\right), & n=2k+1.
 \end{cases}
$$
The DVZ transform is the Jacobi matrix $\mathcal{K}_\lambda$ in \eqref{eq:Kl} built out of the above coefficients $\alpha_n$, while Theorem~\ref{thm:ql} provides the corresponding OPRL $q_n^{(\lambda)}$,
$$
 q_n^{(\lambda)}(x) =
 \begin{cases}
 	\kappa_n
	\left[
	\hat{U}_k-\frac{\alpha_{n-1}}{\alpha_n}(x-\lambda+\alpha_n)\hat{U}_{k-1}
	\right],
	& n=2k,
	\\[7pt]
	\kappa_n
	\left[
	(x-\lambda-\alpha_{n-1})\hat{U}_k-\frac{\alpha_{n-1}}{\alpha_n}\hat{U}_{k-1}
	\right],
	& n=2k+1.
 \end{cases}
$$
To get these expressions we have taken into account that $1+\alpha_{n-1}=\alpha_{n-1}/\alpha_n$.
The measure $\nu_\lambda$ which makes $q_n^{(\lambda)}$ orthonormal follows from \eqref{eq:nul-acpp} and \eqref{eq:mass}, which lead to
$$
 d\nu_\lambda(x) =
 \frac{1-m}{2\pi|\lambda|}\sqrt{\frac{(x+\lambda)^2-1}{1-(x-\lambda)^2}}\,dx +
 m\,\delta(x-1-\lambda)\,dx,
 \qquad x\in E_\lambda.
$$

\subsection{Second kind polynomials of the previous example}
\label{ssec:ex-2mp}

The alluded second kind polynomials are those with Verblunsky coefficients
$$
 \alpha_n = -\frac{1}{n+m^{-1}}, \qquad m\in(0,1).
$$
Therefore, $\rho_n$ and $\kappa_n$ are the same as in the previous example, while the Carath\'eodory function $F$ of the orthogonality measure is the inverse of that one for the preceding example, i.e.
$$
 F(z) = \frac{1-z}{1-(1-2m)z}.
$$
Since $F$ has no singularities on $\T$, the related measure is absolutely continuous and is given by
$$
 d\mu(e^{i\theta}) = \re\,F(e^{i\theta})\,\frac{d\theta}{2\pi} =
 \frac{(1-m)(1-\cos\theta)}{(1-2m)(1-\cos\theta)+2m^2}\,\frac{d\theta}{2\pi},
 \qquad \theta\in(-\pi,\pi].
$$
As for the corresponding OPUC, we find that
$$
 \varphi_n(z) = \textstyle
 \kappa_n \left(z^n - \alpha_{n-1} \sum_{j=0}^{n-1}(1+2jm)z^j\right),
$$
because $\phi_n=\kappa_n^{-1}\varphi_n$ solves the recurrence relation \eqref{eq:mRR} for the monic OPUC, as can be proved by induction. Using \eqref{eq:OLPUC-OPUC} we obtain the corresponding OLPUC,
$$
 \chi_n(z) =
 \begin{cases}
 	\kappa_n \left(z^{-k}-\alpha_{n-1}\sum_{j=-k+1}^k(1+2(k-j)m)z^j\right),
	& n=2k,
	\\[7pt]
	\kappa_n \left(z^{k+1}-\alpha_{n-1}\sum_{j=-k}^k(1+2(k+j)m)z^j\right),
	& n=2k+1.
 \end{cases}
$$

The above results seem to be new for $m\ne1/2$, adding an item to the list of OPUC examples given in \cite[Sect.~1.6]{SiOPUC}, where this case is discussed only for $m=1/2$. Furthermore, via DVZ connection, they provide a new OPRL family with Jacobi matrix $\mathcal{K}_\lambda$ as in \eqref{eq:Kl}, which differs from that one arising in Example~\ref{ssec:ex-mp} only in the main diagonal (up to the first one, the diagonal coefficients of $\mathcal{K}_\lambda$ are the opposite of those in the analogous matrix for Example~\ref{ssec:ex-mp}). According to Theorem~\ref{thm:ql} and using that $1+\alpha_{n-1}=\alpha_{n-1}/\alpha_{n-2}$, we find that the corresponding OPRL are
$$
 q_n^{(\lambda)}(x) =
 \begin{cases}
 	\begin{aligned}
 	& \kappa_n
	\big[ \textstyle
	\hat{U}_k
	-\frac{\alpha_{n-1}}{\alpha_{n-2}}
	(x-\lambda-m\frac{\alpha_{n-2}}{\alpha_{2n-2}})\hat{U}_{k-1}
	\\
	& \kern100pt \textstyle
	+4m\alpha_{n-1}(x-\lambda-1)\sum_{j=0}^{k-2}(j+1)\hat{U}_j
	\big],
	\end{aligned}
	& n=2k,
	\\[9pt]
	\begin{aligned}
 	& \kappa_n
	\big[ \textstyle
	(x-\lambda-\alpha_{n-1})\hat{U}_k
	\\
	& \kern40pt \textstyle
	-\frac{\alpha_{n-1}}{\alpha_{n-2}}
	(1+2m(1+(n-1)(x-\lambda)))\alpha_{n-2})\hat{U}_{k-1}
	\\
	& \kern100pt \textstyle
	-4m\alpha_{n-1}(x-\lambda-1)\sum_{j=0}^{k-2}(j+1)\hat{U}_j
	\big],
	\end{aligned}
	& n=2k+1.
 \end{cases}
$$
The relation \eqref{eq:nul-acpp} implies that they are orthonormal with respect to the measure
$$
 d\nu_\lambda(x) =
 \frac{1}{2\pi|\lambda|}\sqrt{\frac{(x+\lambda)^2-1}{1-(x-\lambda)^2}}
 \frac{(1-m)|x^2-(1+\lambda)^2|}{(1-2m)|x^2-(1+\lambda)^2|+4m^2|\lambda|}\,dx,
 \qquad x\in E_\lambda,
$$
as follows by noticing that, for $\theta=\theta(x)$ as in \eqref{eq:t-x},
$$
 1-\cos\theta = \frac{|x^2-(1+\lambda)^2|}{2|\lambda|}.
$$

\eject

%%%%%%%%%%%%%%%%%%%%%%%%%%%%%%%%%%%%%%%%%%%%%%%%%%%%%%%%%%%%%%%%%%%%%%%%%%
\noindent {\bf \large Acknowledgements}
%%%%%%%%%%%%%%%%%%%%%%%%%%%%%%%%%%%%%%%%%%%%%%%%%%%%%%%%%%%%%%%%%%%%%%%%%%

\bigskip

The work of the first, third and fourth authors has been supported in part by the research project MTM2017-89941-P from Ministerio de Econom\'{\i}a, Industria y Competitividad of Spain and the European Regional Development Fund (ERDF), by project UAL18-FQM-B025-A (UAL/CECEU/FEDER) and by project E26\_17R of Diputaci\'on General de Arag\'on (Spain) and the ERDF 2014-2020 ``Construyendo Europa desde Arag\'on".

The work of the second author has been partially supported by the research project PGC2018--096504-B-C33 supported by Agencia Estatal de Investigaci\'on of Spain.

\bigskip

%%%%%%%%%%%%%%%%%%%%%%%%%%%%%%%%%%%%%%%%%%%%%%%%%%%%%%%%%%%%%%%%%%%%%%%%%%

\end{document}